\documentclass[10pt, a4paper, titlepage]{article}

\usepackage{answers}
\usepackage{setspace}
\usepackage{graphicx}
\usepackage{enumitem}
\usepackage{multicol}
\usepackage{mathrsfs}
\usepackage{amsmath,amsthm,amssymb}
\usepackage{subcaption}
\usepackage{CJKutf8}
\usepackage{float}
\usepackage{mathtools}
\usepackage{MnSymbol}
\usepackage[pdftex,pdfpagelabels]{hyperref}
\usepackage{pst-node}
\usepackage[section]{placeins}
\usepackage{arxiv}

\newtheorem{theorem}{Theorem}[section]
\newtheorem{algorithm}{Algorithm}[section]
\newtheorem{definition}{Definition}[section]
\newtheorem{proposition}{Proposition}[section]

\newtheorem{corollary}{Corollary}[section]
\newtheorem{manualtheoreminner}{Theorem}
\newtheorem{manualdefinitioninner}{Definition}

\newtheorem{manualalgorithminner}{Algorithm}

\newtheorem{remark}{Remark}[section]

\newenvironment{manualdefinition}[1]{%
  \renewcommand\themanualdefinitioninner{#1}%
  \manualdefinitioninner
}{\endmanualdefinitioninner}

\newenvironment{manualalgorithm}[1]{%
  \renewcommand\themanualalgorithminner{#1}%
  \manualalgorithminner
}{\endmanualtheoreminner}

\title{An Algorithm taking Kirby diagrams to Trisection diagrams}

\renewcommand{\shorttitle}{\textit{arXiv} Template}

\hypersetup{
pdftitle={An Algorithm taking Kirby diagrams to Trisection diagrams},
pdfsubject={q-bio.NC, q-bio.QM},
pdfauthor={David S.~Hippocampus, Elias D.~Striatum},
pdfkeywords={First keyword, Second keyword, More},
}

\author{{\hspace{1mm}Willi Kepplinger} \\
	Department of Mathematics\\
	University of Vienna\\
	Oskar-Morgenstern-Platz 1, 1090 Vienna \\
	\texttt{willi.kepplinger@univie.ac.at} \\
}

\begin{document}
\maketitle
 \hypersetup{pageanchor=false}


\thispagestyle{empty}

 
 \thispagestyle{empty}

 
 \hypersetup{pageanchor=true}
 \renewcommand{\thepage}{ \arabic{page} }
 
 \setcounter{page}{1}
 \onehalfspacing

\begin{abstract}
We present an algorithm taking a Kirby diagram of a closed oriented $4$-manifold to a trisection diagram of the same manifold. This algorithm provides us with a large number of examples for trisection diagrams of closed oriented $4$-manifolds since many Kirby-diagrammatic descriptions of closed oriented $4$-manifolds are known. That being said, the algorithm does not necessarily provide particularly efficient trisection diagrams. We also extend this algorithm to work for the non-orientable case.
\end{abstract}

\section{Introduction}
\label{section:introduction}
Kirby diagrams are a classical way of diagrammatically describing closed  $4$-manifolds that has been hugely successful and that is well understood. Their usefulness lies in the fact that they fully describe a closed $4$-manifold using only a framed link and pairs of identified $B^3$'s in $\mathbb{R}^3$. A newer way of describing closed $4$-manifolds using only a genus $g$ surface and $3$ systems of $g$ simple closed curves on said surface, which has become known as trisection diagrams, was introduced by Robion Kirby and David Gay \cite{Gay2016}. This way of looking at $4$-manifolds has since had many applications such as a new way of studying embedded surfaces in $4$-manifolds \cite{Meier2018} or to re-prove the Thom conjecture using purely combinatorial means \cite{LambertCole2020}.\par

The goal of the present work is to answer the obvious question of how to pass between these two diagrammatic description of $4$-manifolds. While it is well known how to pass from a trisection diagram to a Kirby diagram, assuming that the trisection diagram is in a particularly nice form, the other direction has so far not been described in the literature. We thus present an algorithm that takes a Kirby diagram and transforms it into a trisection diagram. Since there is a wealth of examples of Kirby diagrams this provides us with a large number of examples of trisection diagrams though the algorithm may not always produce simple trisection diagrams.\par

Recently Maggie Miller and Patrick Naylor \cite{Miller2020} have worked out the theory of trisections and trisection diagrams for non-orientable $4$-manifolds and so we have included a modified version of the main result, an algorithm that takes a Kirby diagrammatic description of a non-orientable $4$-manifold and gives back a trisection diagram of this $4$-manifold.\par

We now introduce the basic language of trisections in the orientable case. Suitable modifications for the non-orientable case will be made in section \ref{sec:nonorientable}.\par

Given a closed oriented $4$-manifold $X$ we may choose a handle decomposition that is ordered by index and which has only a single $0$-handle and a single $4$-handle. We may therefore build our manifold by starting with this single $0$-handle, attaching the $1$-handles and the $2$-handles, after which we only need to attach a copy of $\natural_m\mathbb{S}^1\times B^3$, i.e. the union of the $m$ $3$-handles and the single $4$-handle. The Laudenbach-Poenaru Theorem \cite{Laudenbach1972} allows us to ignore the specific way in which this copy of $\natural_m\mathbb{S}^1\times B^3$ is attached as any diffeomorphism of $\partial\natural_m\mathbb{S}^1\times B^3=\sharp_m \mathbb{S}^1\times\mathbb{S}^2$ will extend over $\natural_m\mathbb{S}^1\times B^3$. Consequently just prescribing the way in which the $1$- and $2$-handles are attached is enough to specify the manifold provided that $\partial X^{(2)}=\partial \big(0-handle \cup 1-handles \cup 2-handles\big)=\sharp_m \mathbb{S}^1\times\mathbb{S}^2$ for some $m$. Thus a closed oriented $4$-manifold is completely specified by the attaching regions of the $1$-handles and the framed attaching link of the $2$-handles. Projecting these attaching regions living in $\mathbb{R}^3\subset\mathbb{S}^3=\partial B^4$ to an appropriate plane so that the projections of attaching regions of $1$-handles are disjoint, the projections of attaching regions of $1$- and $2$-handles only intersect if the attaching regions in $R^3$ did, and so that the projection of the attaching links of the $2$-handles contains double points at worst, gives a Kirby diagram of the manifold.\par

Since trisection diagrams are less well known we will take more care in introducing them. We begin by defining Trisections.
\begin{definition}{[Trisections]}
\label{def:trisections}
Let $X$ be a closed oriented $4$-manifold. A $(g;\,k_1,k_2,k_3)$-trisection of $X$ is a decomposition 
\begin{align*}
X=X_1\cup X_2\cup X_3
\end{align*} 
where $X_i\cong \natural_{k_i}\mathbb{S}^1\times B^3$, $H_{i,i+1}= X_i\cap X_{i+1}\cong \natural_{g}\mathbb{S}^1\times B^2$, and $\Sigma = X_1\cap X_2\cap X_3\cong\sharp_{g}\mathbb{S}^1\times \mathbb{S}^1$. We call the tuple $(X,X_1,X_2,X_3)$ a trisected manifold.
\end{definition}

Trisections for which $k_1=k_2=k_3$ are called balanced trisections. There is a set of three operations, called $i$-stabilizations where $i$ takes values in $\{1,2,3\}$, that transform a $(g;\,k_1,k_2,k_3)$-trisection of $X$ into a $(g+1;\,k^{\prime}_1,k^{\prime}_2,k^{\prime}_3)$-trisection of $X$, with $k^\prime_i=k_i+1$ and $k^\prime_j=k_j$ for $j\neq i$. This means that an unbalanced trisection may always be stabilized to a balanced one.\par
One of the main reasons why trisections are a nice structure is because they induce diagrammatic descriptions of $4$-manifolds called trisection diagrams that are similar to Heegaard diagrams in dimension $3$. To properly define trisection diagrams we need the following notion.

\begin{figure}[!htb]
	\centering
	\includegraphics[width=0.6\linewidth]{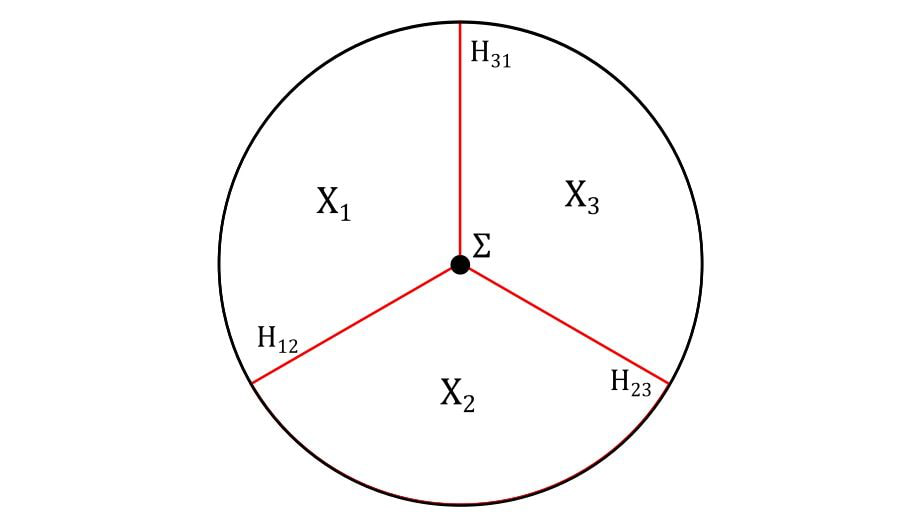}
	\caption{Cartoon of a trisection}
	\label{intro_trisection}
\end{figure}

\begin{definition}
Let $(\Sigma_g,\alpha,\beta)$ and $(\Sigma^\prime_g,\alpha^\prime,\beta^\prime)$ be two surfaces with two systems of simple closed curves each. We say that $(\Sigma_g,\alpha,\beta)$ is slide diffeomorphic to $(\Sigma^\prime_g,\alpha^\prime,\beta^\prime)$ if, after a diffeomorphism taking $\Sigma_g$ to $\Sigma^\prime_g$, $\alpha$ is slide equivalent to $\alpha^\prime$, and $\beta$ to $\beta^\prime$. By $\alpha$ being slide equivalent to $\alpha^\prime$ we mean that there exists a sequence of isotopies and $2$-dimensional handle slides that takes $\alpha$ to $\alpha^\prime$.
\end{definition}

With this notion in mind we can define trisection diagrams.

\begin{definition}{[trisection diagrams]}\label{def:trisectiondiagram}
A $(g;\,k_1,\,k_2,\,k_3)$-trisection diagram $(\Sigma_g,\alpha,\beta,\gamma)$ is a genus $g$ surfaces $\Sigma_g$ together with three sets of $g$ curves $\alpha=(\alpha_1,\dots,\alpha_g)$, $\beta=(\beta_1,\dots,\beta_g)$, and $\gamma=(\gamma_1,\dots,\gamma_g)$ satisfying that 
\begin{itemize}
	\item $(\Sigma_g,\,\alpha,\,\beta)$ is slide diffeomorphic to the $g-k_1$ times stabilized standard Heegaard diagram\\ of $\sharp_{k_1}\mathbb{S}^1\times\mathbb{S}^2$,
	see Figure \ref{fig:trisection_diagram_X_1},
	\item $(\Sigma_g,\,\beta,\,\gamma)$ is slide diffeomorphic to the $g-k_2$ times stabilized standard Heegaard diagram\\ of $\sharp_{k_2}\mathbb{S}^1\times\mathbb{S}^2$,
	\item 
	$(\Sigma_g,\,\gamma,\,\alpha)$ is slide diffeomorphic to the $g-k_3$ times stabilized standard Heegaard diagram\\ of $\sharp_{k_3}\mathbb{S}^1\times\mathbb{S}^2$. 
\end{itemize} 

We say that a trisection diagram is in $X_1$-standard form if $(\Sigma_g,\,\alpha,\,\beta)$ is as depicted in Figure \ref{fig:trisection_diagram_X_1}, so the standard $g-k_1$-times stabilized Heegaard diagram of $\sharp_{k_1}S^1\times S^2$.
$X_2$- and $X_3$-standard form are defined similarly only with $(k_2,\,\beta,\,\gamma)$ and $(k_3,\,\gamma,\,\alpha)$ instead of $(k_1,\,\alpha,\,\beta)$, respectively. 
\end{definition}

\begin{figure}[!htb]
	\centering
	\includegraphics[width=0.8\linewidth]{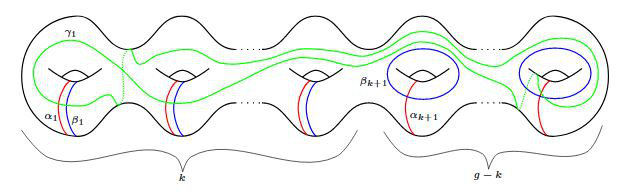}
	\caption{Trisection diagram in $X_1$-standard form. We use the convention of colouring the $\alpha$-curves in red, the $\beta$-curves in blue, and the $\gamma$-curves in green. Note that only one of the $g$ green $\gamma$ curves is drawn.}
	\label{fig:trisection_diagram_X_1}
\end{figure}

Given a trisection of a closed oriented $4$-manifold $X$ we obtain a trisection diagram of $X$ by picking a Morse function that has only a single index $3$ critical point, some index $2$ critical points and no critical points of index $1$ or $0$  on each of the $3$-handlebodies $H_{i,i+1}$. With some choice of pseudo-gradient for these Morse functions we may intersect the resulting descending manifolds of the index $2$ critical points with the central surface $\Sigma$. This gives a system of curves $\alpha,\beta,\gamma$, resulting from $H_{12}, H_{31}, H_{23}$ respectively, so that each pair of these systems of curves is a Heegaard diagram for one of the $\partial X_i\cong \sharp_{k_i}\mathbb{S}^1\times \mathbb{S}^2$. Applying Waldhausens Theorem \cite{Waldhausen1968}, that is the fact that $\partial X_i\cong \sharp_{k_i}\mathbb{S}^1\times \mathbb{S}^2$ only has a single genus $g$ Heegaard diagram up to isotopy and handle slides, we conclude that this process does indeed yield a trisection diagram.\par
Trisection diagrams really only become interesting objects because they turn out to fully capture the information of closed oriented $4$-manifolds:
Let $H_\alpha, H_\beta, H_\gamma$ be the $3$-dimensional handlebodies specified by the $\alpha, \beta, \gamma$-curves respectively. Then, starting from a trisection diagram, we thicken up the surface $\Sigma_g$ to $\Sigma_g\times B^2$ and attach $H_\alpha\times I$ to $\Sigma_g\times I\subset\partial \Sigma_g\times B^2$ along $H_\alpha\times I$, and then do the same for $H_\beta\times I$ and $H_\gamma\times I$ only for different segments on $\partial B^2$. The resulting object is (up to smoothing of corners) a $4$-dimensional manifold with three boundary components, each diffeomorphic to $\sharp_{k_i}\mathbb{S}^1\times \mathbb{S}^2$. The Laudenbach-Poenaru Theorem \cite{Laudenbach1972} now guarantees that up to diffeomorphism there is only a single closed $4$-manifold that is obtained by gluing $\natural_{k_i}\mathbb{S}^1\times B^3$ to the corresponding boundary components, so we obtain a uniquely defined $4$-manifold.\par

The main result the author wishes to present is an explicit algorithm that takes in a Kirby diagram describing a closed oriented $4$-manifold $X$ and produces a trisection diagram describing $X$.

\begin{algorithm}\label{alg:main}
Given a Kirby diagram describing a closed oriented $4$-manifold $X$, which comprises $k$ $1$-handles and a framed link $L$ consisting of $l$ knots $K_1,\dots,K_l$, one obtains a trisection diagram describing $X$ by applying the following algorithm (many of whose steps are illustrated in Figure \ref{intro_steps}):
\begin{enumerate}
    \item Replace the projection of the attaching regions of the $4$-dimensional $1$-handles by attaching regions of $3$-dimensional $1$-handles (a pair of $2$ disks) with a parallel $\alpha-\beta$ curve pair drawn around one of the disks.  
    \item For every knot $K$ in the link $L$ that does not have an overcrossing, introduce a $+1$- then a $-1$-kink. In other words we do a Reidemeister-$2$ move between two strands of the same knot $K$.
    \item Introduce a number of $\pm 1$ kinks equal to the sign and numerical value of the blackboard framing.
    \item Replace all intersections in the fashion shown in Figure \ref{intro_steps}. Note that this increases the number of both $\alpha$- and $\beta$-curves from $l$ to some $g\in \mathbb{N}$.
    \item Do handle slides among the $\alpha$-curves such that $\lvert \alpha_i \cap K_j \rvert = \delta_{ij}$ for $1\leq i \leq g$ and $1\leq j \leq l$.
    \item Rename $K_j=\gamma_j$ for $1\leq j \leq l$ (and color them green).
    \item Draw a green curve parallel to every unpaired $\alpha$-curve, i.e. all the $\alpha$-curves that do not intersect any of the $K_j$ (so precisely $\alpha_{l+1},\dots,\alpha_g$), and call these new green curves $\gamma_{l+1},\dots,\gamma_g$.
\end{enumerate}
\end{algorithm}

\section*{Acknowledgements}
I would like to warmly thank my advisor Vera V\'ertesi for suggesting this nice topic and for many interesting and helpful conversions during the course of this project.
I am also deeply grateful to Michael Eichmair for his constant encouragement and helpful mentoring.
Finally I want to express my gratitude toward David Gay who not only introduced trisections but also was kind enough to look over the completed algorithm and give me extremely generous feedback.\par
Furthermore I want to thank the Vienna School of Mathematics (VSM) for providing a stable and pleasant environment in which to do research.\par
This research was supported by the Austrian Science fund FWF via the START Price project Y-963.

\section{Preliminaries}
The purpose of this section is to recall the existence proof of trisections using handle decompositions as presented in \cite{Gay2016}, with a focus on the construction of $X_2$ as this will be crucial in proving Algorithm \ref{alg:main}.\par
Let $X$ be a closed oriented $4$-manifold with a handle decomposition using one $0$-handle, $k$ $1$-handles, $l$ $2$-handles, $m$ $3$-handles, and one $4$-handle. One can easily see two pieces of the trisection by grouping together the $0$-handle with the $1$-handles and the $4$-handle with the $3$-handles into a $4$-dimensional handlebody each. The difficulty, then, lies in showing how one can use everything that is left to construct another $4$-dimensional handlebody and that all these handlebodies intersect so as to define an unbalanced trisection. We can then stabilize the unbalanced trisection to pass to a balanced one if we wish to. 
Collecting the $0$-handle and the $1$-handles into $X_1$ we consider the unique (that is unique up to isotopy) genus $k$ Heegaard-splitting $\partial X_1=H_{12}\cup_{\Sigma} H_{31}$ with Heegaard-surface $\Sigma\cong\sharp_k\mathbb{S}^1\times\mathbb{S}^1$. We now manipulate the framed attaching link $L$ of the $4$-dimensional $2$-handles and the Heegaard surface $\Sigma$ in several steps:
\begin{enumerate}
    \item We flow $L$ onto $\Sigma$, for example using the flow of a Morse function inducing the Heegaard splitting associated to $\Sigma$, which we can do since $L$ only needs to be disjoint from the descending manifolds of the $3$-dimensional $1$-handles and the ascending manifolds of the $3$-dimensional $2$-handles. 
    \item This flow can be chosen such that the image of $L$ on the surface $\Sigma$ (which we also call $L$) has at worst transverse double point intersections with itself. If a knot in $L$ is intersection free we introduce a kink (i.e. we do a Reidemeister $1$-move on it), achieving that all knots $K_i$ in $L$ have at least one intersection with either itself or another knot $K_j$ in $L$.
    \item As there are only finitely many self-intersection points of $L$ on $\Sigma$ we can stabilize $\Sigma$ near the intersection points in order to resolve them. More precisely we modify the Heegaard surface $\Sigma$ and the corresponding flow slightly using stabilizations so that the flow taking $L$ to the stabilized version of $\Sigma$, which we will continue calling $\Sigma$, is an isotopy and therefore does not introduce any intersections. Note that the genus of the handlebodies $H_{12}$ and $H_{31}$ also increases as we stabilize $\Sigma$. For an illustration of this process see Figure \ref{resolv_intersec}.
    \item We introduce kinks and resolve them with stabilizations until the framing of $L$ coincides with the surface framing. In other words this step achieves that the link $L$ is zero-framed with respect to the Heegaard-surface $\Sigma$. The new genus of $\Sigma$ is denoted with $g$, meaning we have stabilized a total of $g-k$ times. 
    \item We consider a complete system of compression disks $D_j$ for the handlebody $H_{12}$ and call $H_{12}\cap D_j=\alpha_j$. We know that every knot $K_i$ intersects one of the $\alpha_j$: recall that every knot $K_i$ runs over a copy of $S^1\times S^1$ that no other knot runs over since at least one copy was created specifically to resolve an intersection of $K_i$. Furthermore we can use handle slides among the compression disks to ensure that every $K_i$ intersects only one of the $\alpha_j$. Notice that since this process associates to every $K_i$, of which there are $l$ many, exactly one $\alpha_j$, of which there are $g$ many, there are $g-l$ leftover $\alpha$-curves that do not intersect any of the $K_i$.
\end{enumerate}

\begin{figure}[!htb]
	\centering
	\includegraphics[width=0.6\linewidth]{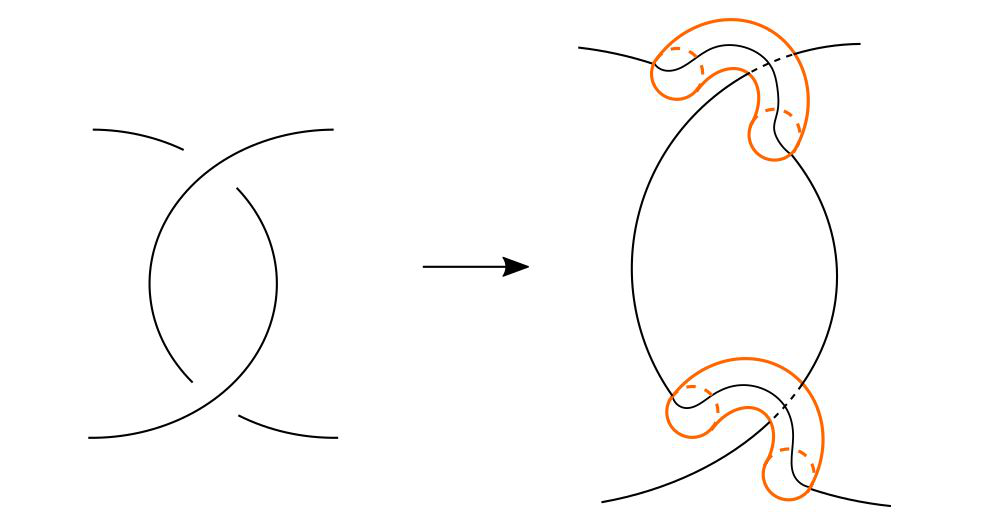}
	\caption{Resolving intersections}
	\label{resolv_intersec}
\end{figure}

We are now ready to construct a trisection of $X$:\par
We set $X_1=\text{0-handle}\cup\text{1-handles}$, 
$X_2=\mathcal{N}(H_{12}) \cup_L \big(\text{4-dimensional 2-handles}\big)$, meaning that we take a closed collar neighbourhood $\mathcal{N}(H_{12})\cong H_{12}\times I$ of $H_{12}$ inside $X$ such that $H_{12}\times \{0\}\subset \partial X_1$ and attach the $4$-dimensional $2$-handles of the handle decomposition of $X$ along $L\times \{1\}$. Finally we define $X_3=X\setminus (X_1\cup X_2)$. Evidently $X_1$ and $X_3$ are handlebodies, it therefore remains to show that $X_2$ is a $4$-dimensional handlebody and that all intersections of the $X_i$ are as advertised.

\begin{figure}[!htb]
	\centering
	\includegraphics[width=0.3\linewidth]{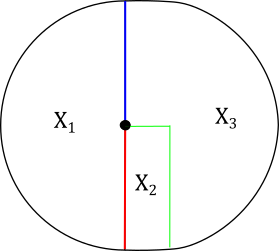}
	\caption{Schematic of the trisection}
	\label{sketch_trisection}
\end{figure}

\begin{theorem}{[Existence of trisections for closed oriented $4$-manifolds]}
\label{thm:existence}
$X=X_1\,\cup X_2\,\cup X_3$ is a trisection.
\end{theorem}
\begin{proof}
A crucial observation is that since the framing of $L$ equals the surface framing we can think
of \\ $H_{12}\times I \cup_{L}\big(\text{4-dimensional 2-handles}$ with framing induced by $\Sigma\big)$ as 
$\big(H_{12} \cup_{L}\text{3-dimensional 2-handles}\big)\times I$.

Having arranged $L$ and the compression disks $D_i$ as we did, attaching $3$-dimensional $2$-handles along $L\subset \Sigma = \partial H_{12}$ leads to these $l$ $2$-handles canceling with $l$ $3$-dimensional $1$-handles of $H_{12}$. This is because for each $K_i\in L$, the attaching sphere of the corresponding $3$-dimensional $2$-handle, intersects the belt sphere of precisely one $3$-dimensional $1$-handle transversely and exactly once, so $l$-many of the $g$ $1$-handles are cancelled. Crossing the three dimensional handlebody $H_{12}$ with $I$ we get $H_{12}\times I\cong \natural_g\mathbb{S}^1\times B^3$ and see that the co-cores $D_j$ of the $3$-dimensional $1$-handles become a compression balls $B_j=D_J\times I$ of $4$-dimensional $1$-handles. The boundary $\partial B_j\cong\mathbb{S}^2$ correspond to the belt spheres of these $4$-dimensional $1$-handles. We interpret $L$ as the attaching link of $4$-dimensional $2$-handles (with the framing induced by the surface $\Sigma_{g}$), push all its constituent knots $K_i$ into $H_{12}\times \{1\}$ so that each of them intersects a compression disk $D_i\times \{1\}$ transversely and exactly once. Since this compression disk is part of the boundary of the compression ball $B_i$ and is the only place in which the attaching sphere $K_i$ intersects said belt sphere we again end up with a cancellation. Thus $X_2$ is a $4$-dimensional handlebody of genus $g^\prime=g-l$.\par

Next, a quick look at Figure \ref{sketch_trisection} tells us that 
\begin{itemize}
	\item $X_1\cap X_2=H_{12}\cong\natural_{g}\mathbb{S}^1\times B^{2}$
	\item $X_3\cap X_1=H_{31}\cong\natural_{g}\mathbb{S}^1\times B^{2}$
	\item $X_1\cap X_2\cap X_3\cong \Sigma_g$
\end{itemize} 
and so the only intersection left to check is $X_2\cap X_3=\partial \overline{X_2\setminus H_{12}}$. Denote by $C_L$ the $3$-dimensional compression body defined by $L$, obtained by attaching $l$ $1$-handles to $\Sigma_{g}$ so that their belt spheres coincide with $L$ on $\Sigma_g$, and denote by $C_L^\ast$ its dual compression body, i.e. the upside down version of $C_L$. Going forward we will also denote the dual of a handlebody $H$ by $H^\ast$. Note that using this notation $H_{12}\cup_L C_L^\ast$ is the handlebody obtained by attaching $3$-dimensional $2$-handles to $H_{12}$ along $L$ and that we may write $X_2=\big(H_{12}\cup_L C_L^\ast\big)\times I$ with boundary as depicted in Figure \ref{X_2_int_X_3}. \par

\begin{figure}[!htb]
	\centering
	\includegraphics[width=0.6\linewidth]{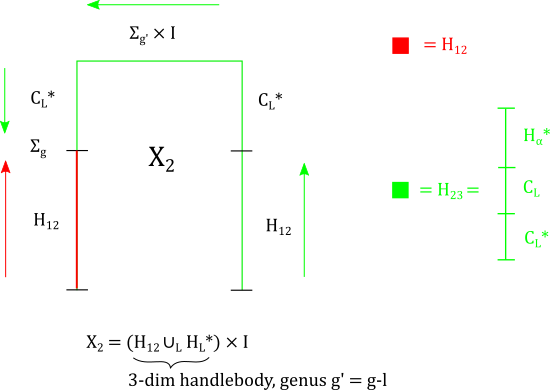}
	\caption{Finding $X_2\cap X_3=H_\gamma$. By $\Sigma_{g^\prime}$ we mean $\partial\big(H_{12}\cup_L C_L^\ast\big)$.}
	\label{X_2_int_X_3}
\end{figure}

Since $X_2$ is glued to $X_1$ along $H_{12}$, $X_2\cap X_3=C_L^\ast \cup \big(H_{12}\cup_{L} C_L^\ast\big)^\ast$, which is just a $3$-handle with a bunch of $2$-handles, so a handlebody.
\end{proof}

The significant portion of this proof for the purposes of the present discussion is the way in which $X_2$ and in particular its boundary is built. The algorithm will be constructed by analyzing $H_{23}=\overline{\partial X_2\setminus H_{12}}$ and the attaching regions of the corresponding $3$-dimensional $2$-handles.

\section{Main Result}

Our aim in this section is to present Algorithm $\ref{alg:main}$ which constructs from a Kirby diagram describing $X$ a trisection diagram describing $X$ and to indicate how to go in the other direction as well. All Kirby diagrams in this section will be framed using the blackboard framing.\par

We will start with the easier direction: going from a trisection diagram to a Kirby diagram is not difficult assuming the trisection diagram is in $X_i$-standard position (wlog we will assume it is in $X_1$-standard position).
Given a trisection diagram $(\Sigma_g,\alpha,\beta, \gamma)$ of a closed oriented $4$-manifold $X$ in $X_1$-standard form (i.e. we look at the standard Heegaard diagram for $\alpha$ and $\beta$ as in Figure \ref{fig:trisection_diagram_X_1}) we have the $g$-many $\gamma$-curves that will in general look very complicated. Together with the framing induced by the surface, $\gamma=(\gamma_1,\dots,\gamma_g)$ is a framed link. We will now see that, given a trisection diagram of $X$ in $X_1$-standard position one can already see a Kirby diagram for $X$. The following Algorithm \ref{alg:tris_to_kirby} was already described in \cite{Gay2016}.

\begin{algorithm}
\label{alg:tris_to_kirby}
Let $(\Sigma_g,\alpha,\beta,\gamma)$ be a trisection diagram of a closed oriented $4$-manifold $X$ in $X_1$-standard form. Then one obtains a Kirby diagram describing $X$ by applying the following algorithm
\begin{enumerate}
    \item Replace every parallel $\alpha-\beta$ pair with a dotted circle.
    \item Delete all other $\alpha-\beta$ pairs.
    \item Regard all $\gamma$-curves as a framed link, where the framing is induced by the surface $\Sigma$ (convert this to blackboard framing if needed).
    \item Delete the surface $\Sigma_g$.
\end{enumerate}
\end{algorithm}

\begin{proof}
The $\alpha$ and $\beta$ curves allow us to build $\natural_{k_1} \mathbb{S}^1\times \mathbb{S}^2$ and we may then uniquely glue to this $\sharp_{k_1} \mathbb{S}^1\times B^3$ which we then call $X_1$. The $4$-dimensional handlebody $X_1$ can be interpreted as a $4$-dimensional $0$-handle with $k_i$ $4$-dimensional $1$-handles attached to it. Now we glue $4$-dimensional $2$-handles to $X_1$ along $\gamma$ using the induced surface framing. This is the same as attaching $3$-dimensional $2$-handles and crossing the result with $I$. Since the union of $\Sigma$ and the $3$-dimensional $2$-handles is $\overline{H_{23}\setminus B^3}$, the union of $\Sigma\times I$ and the $4$-dimensional $2$-handles is $\big(H_{23}\setminus B^3\big)\times I$. See the result of this construction in Figure \ref{tris-kirby}. So far we have used only the $0$-, the $1$-, and the $2$-handles so the complement is the union of the $3$-handles and a $4$-handle, so a $4$-dimensional handlebody. The structure which we have built up so far exactly coincides with the data provided by the Kirby diagram defined by the process outlined above, so this Kirby diagram describes the closed oriented $4$-manifold $X$.

\begin{figure}[!htb]
	\centering
	\includegraphics[width=0.5\textwidth]{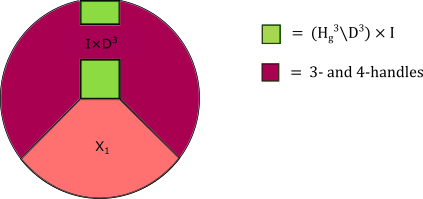}
	\caption{$X_1$ with attached $4$-dimensional $2$-handles}
	\label{tris-kirby}
\end{figure}
\end{proof}

\begin{figure}[!htb]
	\centering
	\includegraphics[width=0.4\linewidth]{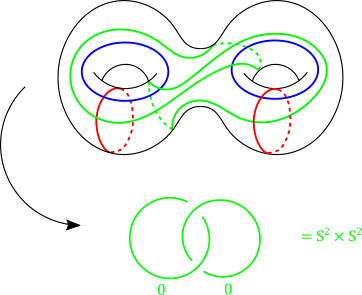}
	\caption{An example of going from a trisection diagram in $X_1$-standard form to a Kirby diagram}
	\label{tris_diagram-kirby}
\end{figure}

\begin{remark}
Algorithm \ref{alg:tris_to_kirby} tells us that we can sometimes guess the trisection diagram of a manifold, provided we know a Kirby diagram describing the manifold. We just need to find a trisection diagram in $X_1$-standard position such that, after applying Algorithm \ref{alg:tris_to_kirby}, we end up with the desired Kirby diagram. The difficulty in guessing the correct trisection diagram is checking that the surface plus collection of curves one arrives at through guesswork really is a trisection diagram. Furthermore one can use Algorithm \ref{alg:tris_to_kirby} to prove that the only closed oriented $4$-manifolds which admit trisections of genus $0$ or $1$ are $\mathbb{S}^4,\mathbb{S}^1\times\mathbb{S}^3$, $\mathbb{C}P^2$, and $\overline{\mathbb{C}P^2}$ by an easy distinction of cases.
\end{remark}

We now want to build up to Algorithm \ref{alg:main}. The next proposition tells us how to construct a trisection diagram from a Kirby diagram, but it is not yet fully algorithmic because it relies on a choice of flow that moves the framed attaching link of the $2$-handles onto a Heegaard surface of $X_1$.\par
Recall that in the proof of the existence of trisections we just picked some Heegaard surface onto which we flowed the attaching link. The idea is that we construct one concretely and then work from there. 

\begin{proposition}
\label{prop:3d_diagram}
Given a Kirby diagram describing a closed oriented $4$-manifold $X$, which comprises $k$ $1$-handles and a framed link $L$ consisting of $l$ knots $K_1,\dots,K_l$ one obtains a trisection diagram $(\Sigma_g,\alpha,\beta,\gamma)$ describing $X$ by applying the following algorithm
\begin{enumerate}
    \item Construct meridians for all $1$-handles and connect them so that they form a bouquet of $k$ circles.
    \item Thicken up this bouquet of $k$ circles to $3$-dimensional genus $k$-handlebody until the dotted circles lie on $\Sigma$, the boundary of this handlebody. Note that the surface $\Sigma$ is a genus $k$ Heegaard-surface for $\partial \big(0-handle \cup 1-handles\big)$ and that the dotted circles form a cutting system for it.
    \item Like in the proof of Theorem \ref{thm:existence}, stabilize $\Sigma$ suitably often so that $L$ can be flowed onto the Heegaard surface without self-intersections and so that the framing agrees with the surface framing. If any of the $K_i$ does not intersect itself or another $K_j$, artificially introduce a self intersection using a Reidemeister I move and resolve it with a stabilization. Let $g$ denote the genus of the stabilized version of the surface $\Sigma$ which we will still denote by $\Sigma$.
    \item Replace all dotted circles with a parallel $\alpha-\beta$ pair and draw a canceling $\alpha-\beta$ pair at the locus of any stabilization.
    \item Rename $K_1,\dots,K_l$ to $\gamma_1,\dots,\gamma_l$. Do handleslides among the $\alpha$-curves and afterward rename them if necessary so that $\lvert\gamma_i\cap\alpha_j\rvert=\delta_{ij}$ for all $1\leq i\leq l,1\leq j\leq g$.
    \item Draw parallel pushoffs for all $\alpha_j$ with $j>l$ and call them $\gamma_j$.
\end{enumerate}
\end{proposition}

\begin{figure}[!htb]
	\centering
	\includegraphics[width=0.6\linewidth]{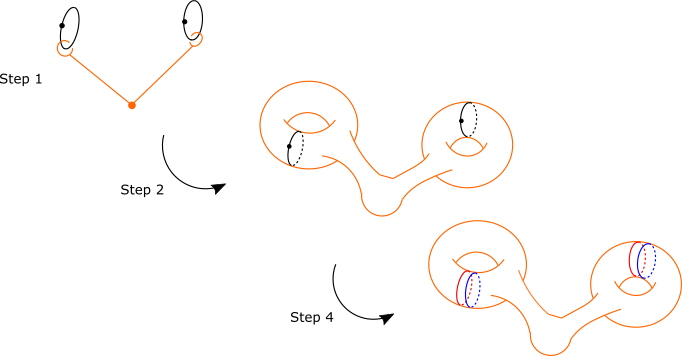}
	\caption{The process of constructing the Heegaard-surface for $2$ dotted circles.}
	\label{algorithm_surface}
\end{figure}

\begin{proof}
The first thing we want to show is that $\Sigma$ actually is a Heegaard surface for $X_1$. The argument can be reduced to a single dotted circle with a meridian $\mu$ around it by looking at just one of the connect summands in $\partial X_1\simeq \sharp_k \mathbb{S}^1\times \mathbb{S}^2$. Now thicken up the meridian and call the resulting genus $1$ handlebody $H$.
The boundary of the $0$-handle with a single attached $1$-handle is $\mathbb{S}^1\times \mathbb{S}^2$, with $\mu$ running along the $\mathbb{S}^1$-factor. The boundary of $H$, being nothing but $\partial (\mathbb{S}^1\times B^2)=\mathbb{S}^1\times \mathbb{S}^1$, separates every $\{pt.\}\times S^2$ into an inner and an outer disk, meaning that this embedded $\mathbb{S}^1\times \mathbb{S}^1$ bounds a solid torus on both sides. We have thus found a Heegaard surface for $\mathbb{S}^1\times \mathbb{S}^2$.\par
After flowing the attaching link $L$ of the $2$-handles onto $\Sigma_g$, with all the steps like introducing and resolving intersections and kinks if necessary, we get $(\Sigma,\alpha,\beta)$ together with $l$ simple closed curves $K_1,\dots,K_l$. The only thing we still have to worry about is finding the $\gamma$-curves. Recall that we already analyzed $H_\gamma=X_2\cap X_3 = C_L^\ast \cup\big(H_{12}\cup C_L^\ast\big)^\ast$ briefly in the proof of Theorem \ref{thm:existence}. The first $2$-handles we attach, the ones defining the compression body $C_L^\ast$, are precisely represented by the attaching link $L\subset \Sigma_g$, so $\gamma_1,\dots,\gamma_l$ are given by $K_1,\dots,K_l$. Now we do handleslides among the $\alpha$-curves so that every $K_i\in L$ intersects exactly one of the $\alpha$-curves. After possibly renaming, we have that $\lvert K_i\cap \alpha_j\rvert=\delta_{ij}$ for $1\leq i\leq l$ and $1\leq j\leq g$. To get to the missing $\gamma$-curves we now inspect $\big(H_{12}\cup C_L^\ast\big)^\ast$. To avoid getting lost in the following argument, consult Figures \ref{X_2_int_X_3} and \ref{fig:doubling}. We recall that $g^\prime=g-l$ and that $\partial X_2 = \sharp_{g^\prime} \mathbb{S}^1\times \mathbb{S}^2=\big(H_{12}\cup C_L^\ast\big)\cup_{\Sigma_{g^\prime}}\big(H_{12}\cup C_L^\ast\big)^\ast$ is a Heegaard splitting. The genus $g^\prime$ handlebody $\big(H_{12}\cup C_L^\ast\big)$ is the handlebody that remains after taking $H_{12}$ and filling in the $2$-handles defined by $L$. To get $\partial X_2$ we simply double this handlebody by attaching its dual to its boundary using the identity map. Diagrammatically this means that we have to attach $2$-handles to curves parallel to the remaining $\alpha$-curves $\alpha_{l+1},\dots,\alpha_g$. These are then the missing $\gamma$-curves.\par
Putting these $g-l$ new curves together with the original $l$ curves $K_1,\dots,K_l$ we now have a complete system of $g$ curves in $\Sigma_g$. Calling them $(\gamma_1,\dots,\gamma_g)=\gamma$, $(\Sigma,\alpha,\beta,\gamma)$ is a trisection diagram by construction. The reader is invited to apply Algorithm \ref{alg:tris_to_kirby} for a sanity check and see that we retrieve the original Kirby diagram plus some $0$-framed unknots which represent cancelling $2-3$ handle pairs.
\end{proof}

\begin{figure}[!htb]
	\centering
	\includegraphics[width=0.3\linewidth]{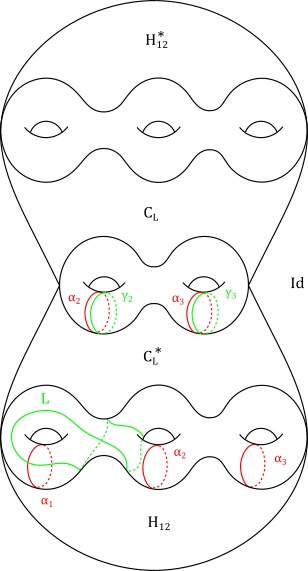}
	\caption{Finding the missing $\gamma$-curves by realizing $\partial X_2$ as the double of $H_{12}\cup C_L^\ast$}
	\label{fig:doubling}
\end{figure}

\begin{remark}
The number of stabilizations in Proposition \ref{prop:3d_diagram} can be reduced by observing that, after step $3$, every $K_i$ has a stabilization of its own, or rather $3$-dimensional $1$-handle that $K_i$ runs over but that no other $K_j$ runs over. The framing of $K_i$ can now be made to coincide with the surface framing not with the application of Reidemeister $1$-moves and subsequent stabilizations but by winding $K_i$ around the boundary of this $3$-dimensional $1$-handle.  
\end{remark}

Using Proposition \ref{prop:3d_diagram} we are in a position to prove Algorithm \ref{alg:main}. The trick of the proof is to regard the blackboard on which the Kirby diagram is defined as (part of) the Heegaard surface found in Proposition \ref{prop:3d_diagram}. This serves two functions: first it means that we do not have to choose any flows and second it means that the notions of blackboard framing and surface framing coincide. So far we only ever argued that the surface framing can be made to coincide with the framing of the attaching link but this was not enough to get an algorithm since an algorithm has to prescribe exactly how many kinks one has to introduce, not just that there is a number of kinks that makes the framings coincide. The following algorithm thus is an adaptation of Proposition \ref{prop:3d_diagram}.

\begin{proof}[Proof of Algorithm \ref{alg:main}]
If we can show that a Kirby diagram can be interpreted as a link laying on a Heegaard surface then we are done since all the steps we do afterwards just precisely mirror those of Proposition \ref{prop:3d_diagram}. It is not difficult to reinterpret Kirby diagrams in this way: Recall that any Kirby diagram is obtained by projecting the attaching regions of the $4$-dimensional $1$-handles and the framed attaching link of the $4$-dimensional $2$-handles living in $\mathbb{R}^3$ onto a suitable plane $P\subset \mathbb{R}^3$. Viewing $P$ as the boundary of a $3$-dimensional $0$-handle and the projections of the attaching regions of the $4$-dimensional $1$-handles as attaching regions of $3$-dimensional $1$-handles, we can construct a Heegaard splitting of $\natural_k \mathbb{S}^1\times\mathbb{S}^2$, where $k$ is the number of $4$-dimensional $1$-handles, by letting the $3$-dimensional $1$-handles run over the $4$-dimensional $1$-handles. Thus the Kirby diagram really is a diagram on $\mathbb{R}^2$ with the blackboard framing coinciding with the surface framing of a Heegaard-surface. This process is illustrated in Figure \ref{fig:projection}.\par 

\begin{figure}[!htb]
	\centering
	\includegraphics[width=0.5\linewidth]{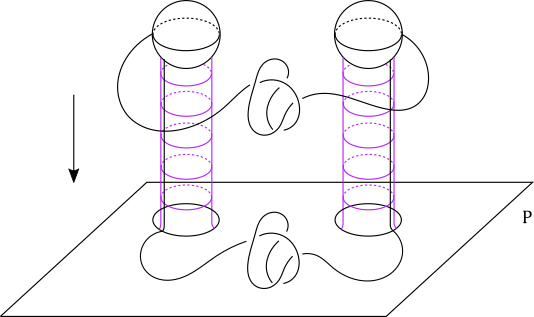}
	\caption{An illustration of the projection that allows one to interpret the Kirby diagram as a link almost laying on a Heegaard surface.}
	\label{fig:projection}
\end{figure}

The only reason why step $2$ introduces $2$ self intersections is to ensure that the framing of the knot in question does not change which it would if we only introduced one self intersection.\par
The process up to step $5$ ensures that every knot $K$ in the $2$-handle attaching link has an $\alpha$-curve that it intersects once and transversely and that no other knot intersects. It might still be the case that $K$ intersects more than just one $\alpha$-curve which we want to avoid, necessitating Step 5:\par
Recall that, because of step $2$, every $K_j$ has an overcrossing. In particular every $K_1$ therefore has some $\alpha$-curve that it intersects only once, if there are several pick one and call it $\alpha_1$. Every other $\alpha$-curve that $K_1$ intersects can now be slid over $\alpha_1$ along the segment of $K_1$ that connects them. We have thus achieved that $\lvert \alpha_i\cap K_1\rvert=\delta_{i1}$. Repeat this process for all $j$ and note that doing this process for other knots does not destroy this intersection property for the knots that came before it.\par
The result follows by the same reasoning as in proposition \ref{prop:3d_diagram}.
\end{proof}
For a simple example of this algorithm in action see Figure \ref{application_algorithm}.

\begin{figure}[!htb]
	\centering
	\includegraphics[width=0.6\linewidth]{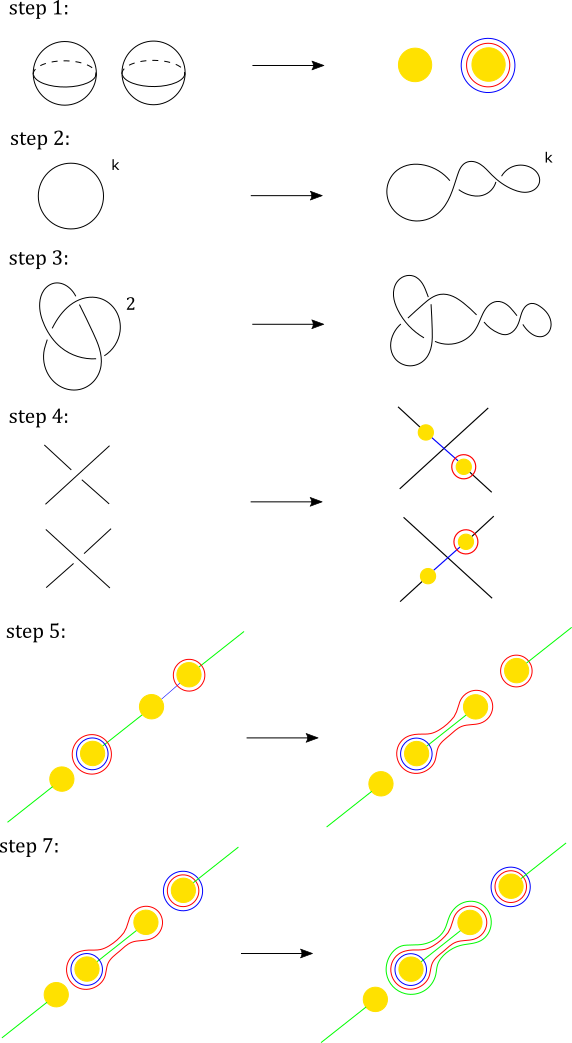}
	\caption{Illustrating some of the steps described in Algorithm \ref{alg:main}}
	\label{intro_steps}
\end{figure}

\begin{figure}[!htb]
	\centering
	\includegraphics[width=0.7\linewidth]{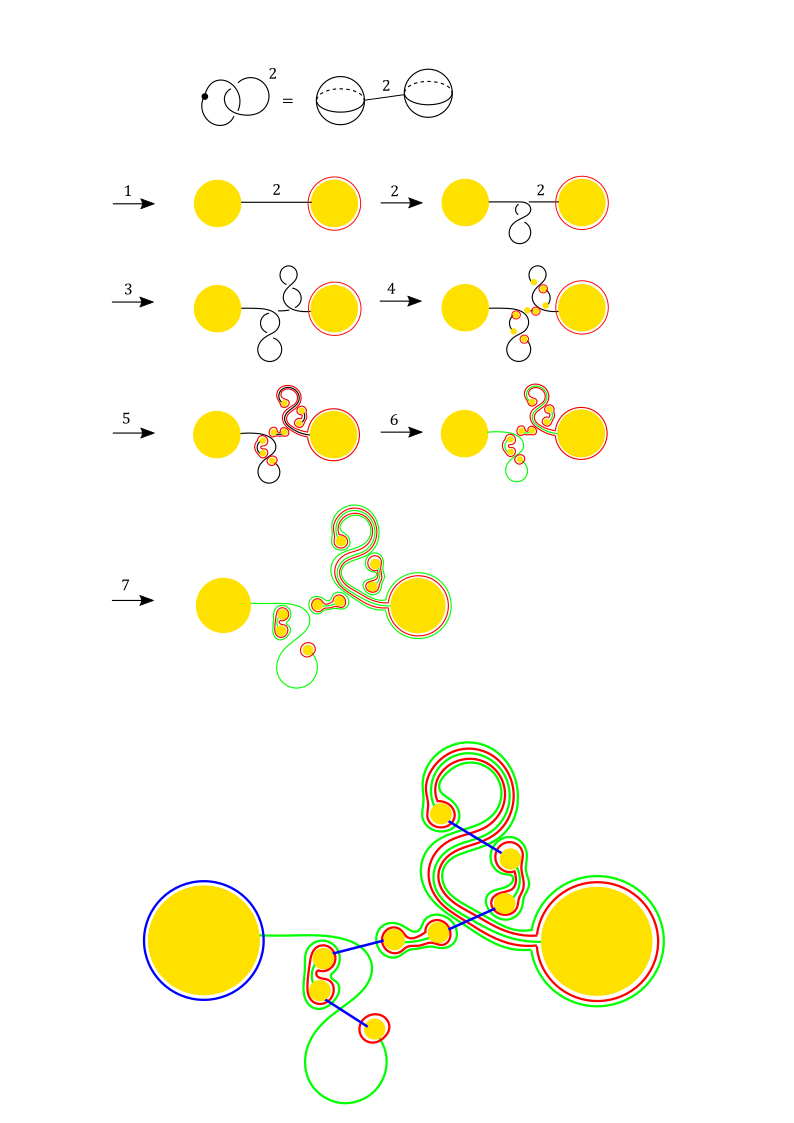}
	\caption{Applying the algorithm to a simple example. The blue curves were suppressed in the steps (since they are all standard anyways and do not feature in the algorithm) and were only added in the end.}
	\label{application_algorithm}
\end{figure}

\begin{remark}
Again, the number of stabilizations in the algorithm can be reduced by observing that, after step $4$ in Algorithm \ref{alg:main}, every $K_i$ has an overcrossing of its own. This means that for every $K_i$ there exists a $3$-dimensional $1$ handle that $K_i$ runs over but that no $K_j$ for $i\neq j$ runs over. The framing of $K_i$ can now be made to coincide with the surface framing not with the introduction of kinks and subsequent stabilizations but by assigning an integer to the handle that is crossed only by $K_i$ which specifies the number of times $K_i$ winds around this handle.
\end{remark}

Applying this slightly streamlined algorithm to the same example as the one portrayed in Figure \ref{application_algorithm} yields the diagram shown in Figure \ref{application_algorithm_simplified}.

\begin{figure}[!htb]
	\centering
	\includegraphics[width=0.3\linewidth]{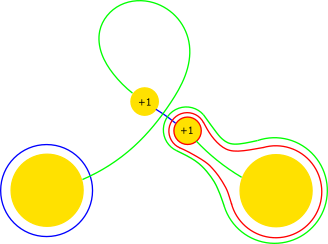}
	\caption{The result of applying the simplified algorithm to the same Kirby diagram as in Figure \ref{application_algorithm}.}
	\label{application_algorithm_simplified}
\end{figure}

\section{The Non-Orientable Case}
\label{sec:nonorientable}
Trisections of non-orientable $4$-manifolds were introduced by Maggy Miller and Patrick Naylor \cite{Miller2020}. We will repeat all relevant definitions here. \par
The natural notion of a non-orientable genus $k$ handlebody is a $0$-handle with $k$ orientation reversing $1$-handles attached, we will denote this object by $\natural_k\mathbb{S}^1\Tilde{\times}\,B^3$. It should be remarked that $\natural_k\mathbb{S}^1\Tilde{\times}\,B^3\cong\mathbb{S}^1\Tilde{\times}\,B^3\natural_{k-1}\mathbb{S}^1\times B^3$ since any of the orientation preserving $1$-handles can be converted to orientation reversing $1$-handles by sliding them over an orientation reversing $1$-handle. Further note that $\partial(\natural_k\mathbb{S}^1\Tilde{\times}\,B^3)=\sharp_k \mathbb{S}^1\Tilde{\times}\,\mathbb{S}^2$, so the $k$-fold connected sum of the non-orientable $\mathbb{S}^2$-bundle over $\mathbb{S}^1$.\par
Now given a handle decomposition of a closed non-orientable $4$-manifold $X$ with a single $0$-handle and a single $4$-handle we can again ask the question whether specifying the way in which the $1$- and $2$-handles are attached is enough to determine $X$ up to diffeomorphism. A non-orientable version of the Laudenbach-Poenaru theorem (proven by Maggie Miller and Patrick Naylor \cite{Miller2020}) answers this in the affirmative.\par
\begin{figure}[!htb]
	\centering
	\includegraphics[width=0.3\linewidth]{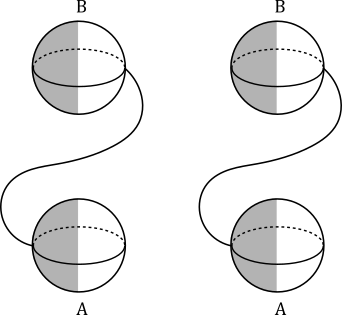}
	\caption{An example of a Kirby diagram with non-orientable $1$-handles. This Kirby diagram describes $\mathbb{S}^1\Tilde{\times}\,\mathbb{S}^3$, albeit not in the most simple way which would just be a single non-orientable $1$-handle with no $2$-handles. The framing of the $2$-handle was not specified because any framing would lead to a description of the same manifold.}
	\label{fig:nonorientableKirby}
\end{figure}
As for the attachment of $4$-dimensional $2$-handles it should be remarked that, since their attaching regions are diffeomorphic to $\mathbb{S}^1\times B^2$, the cores of $2$-handles cannot run over orientation reversing $1$-handles an odd number of times. There is a small issue with framing: the blackboard framing is well defined only if every segment of the attaching link that connects to an orientation reversing $1$-handle is framed separately, see \cite{Miller2020} for details. 
With all of these provisions made we see that Kirby diagrams generalize in a straighforward fashion to the non-orientable case. We now turn to trisections and trisection diagrams of closed non-orientable $4$-manifolds.

\begin{manualdefinition}{\ref{def:trisections}'}\label{def:trisectionsprime}
Let $X$ be a closed non-orientable $4$-manifold. A $(g;\,k_1,k_2,k_3)$-trisection of $X$ is a decomposition 
\begin{align*}
X=X_1\cup X_2\cup X_3
\end{align*} 
where $X_i\cong \natural_{k_i}\mathbb{S}^1\Tilde{\times}\,B^3$, $H_{i,i+1}= X_i\cap X_{i+1}\cong \natural_g\mathbb{S}^1\Tilde{\times}\,B^2$, and $\Sigma = X_1\cap X_2\cap X_3\cong\sharp_g\mathbb{S}^1\Tilde{\times}\,\mathbb{S}^1$. We call the tuple $(X,X_1,X_2,X_3)$ a trisected manifold.
\end{manualdefinition}
We will define the genus of the non-orientable surface $\sharp_g\mathbb{S}^1\Tilde{\times}\,\mathbb{S}^1=:\Tilde{\Sigma}_g$ to be $g$ so that the non-orientable case can be discussed in perfect analogy to the orientable one.\par
The notion of slide-diffeomorphism equivalence does not change at all, i.e. $(\Tilde{\Sigma}_g,\alpha,\beta)$ is slide diffeomorphic to $(\Tilde{\Sigma}^\prime_g,\alpha^\prime,\beta^\prime)$ if there exists a diffeomorphism carrying $\Tilde{\Sigma}_g$ to $\Tilde{\Sigma}^\prime_g$ so that the systems of curves $\alpha$ and $\beta$ are slide equivalent to the systems of curves $\alpha^\prime$ and $\beta^\prime$, respectively.

\begin{manualdefinition}{\ref{def:trisectiondiagram}'}\label{def:trisectiondiagramssprime}
A $(g;\,k_1,\,k_2,\,k_3)$-trisection diagram $(\Tilde{\Sigma}_g,\alpha,\beta,\gamma)$ is a genus $g$ surfaces $\Tilde{\Sigma}_g$ together with three sets of $g$ curves $\alpha=(\alpha_1,\dots,\alpha_g)$, $\beta=(\beta_1,\dots,\beta_g)$, and $\gamma=(\gamma_1,\dots,\gamma_g)$ satisfying that 
\begin{itemize}
	\item $(\Tilde{\Sigma}_g,\,\alpha,\,\beta)$ is slide diffeomorphic to the $g-k_1$ times stabilized standard Heegaard diagram\\ of $\sharp_{k_1}\mathbb{S}^1\times\mathbb{S}^2$,
	see Figure \ref{fig:nonorientable_tris_diag},
	\item $(\Tilde{\Sigma}_g,\,\beta,\,\gamma)$ is slide diffeomorphic to the $g-k_2$ times stabilized standard Heegaard diagram\\ of $\sharp_{k_2}\mathbb{S}^1\times\mathbb{S}^2$,
	\item 
	$(\Tilde{\Sigma}_g,\,\gamma,\,\alpha)$ is slide diffeomorphic to the $g-k_3$ times stabilized standard Heegaard diagram\\ of $\sharp_{k_3}\mathbb{S}^1\times\mathbb{S}^2$. 
\end{itemize} 

We say that a trisection diagram is in $X_1$-standard form if $(\Tilde{\Sigma}_g,\alpha,\beta)$ looks like Figure \ref{fig:nonorientable_tris_diag}, meaning that $(\Tilde{\Sigma}_g,\alpha,\beta,\gamma)$ are as in Figure \ref{fig:trisection_diagram_X_1}. In particular we want to consider $\Tilde{\Sigma}_g$ as $\mathbb{S}^1\Tilde{\times}\,\mathbb{S}^1\natural_{g-1}\mathbb{S}^1\times\mathbb{S}^1$.
$X_2$- and $X_3$-standard form are defined similarly only with $(k_2,\,\beta,\,\gamma)$ and $(k_3,\,\gamma,\,\alpha)$ instead of $(k_1,\,\alpha,\,\beta)$, respectively.
\end{manualdefinition}
It is clear that a non-orientable trisection induces a non-orientable trisection diagram exactly as discussed for orientable trisections. Going in the other direction also works if one uses a non-orientable version of Waldhausens theorem proven in \cite{Miller2020} that posits uniqueness of genus $g$ Heegaard splittings of $\sharp_k \mathbb{S}^1\Tilde{\times}\,\mathbb{S}^2$ up to isotopy and handleslides. 

Clearly it makes sense to ask for an algorithm that relates Kirby- and trisection diagrams in the non-orientable case. Indeed everything works out in precisely the way one would expect. \par
The proof of existence of trisections for closed non-orientable $4$-manifolds can be followed verbatim. One might worry about making sense of the surface framing but since every knot in the attaching link $L$ has to run over orientation reversing $1$-handles an even number of times and since we frame every arc connecting orientation reversing $1$-handles separately this works out fine.\par
The remaining Algorithms and Propositions also have obvious counterparts:

\begin{manualalgorithm}{\ref{alg:tris_to_kirby}'}\label{alg:tris_to_kirby_prime}
Let $(\Tilde{\Sigma}_g,\alpha,\beta,\gamma)$ be a trisection diagram of a closed non-orientable $4$-manifold $X$ in $X_1$-standard form. Then one obtains a Kirby diagram describing $X$ by applying the following algorithm
\begin{enumerate}
    \item replace every parallel $\alpha-\beta$ pair with a $1$-handle, the leftmost one of which is orientation reversing.
    \item delete all other $\alpha-\beta$ pairs.
    \item regard all $\gamma$-curves as a framed link, where the framing is induced by the surface $\Tilde{\Sigma}_g$ (convert this to blackboard framing if needed).
    \item delete the surface $\Tilde{\Sigma}_g$.
\end{enumerate}
\end{manualalgorithm}

\begin{figure}[!htb]
	\centering
	\includegraphics[width=0.8\linewidth]{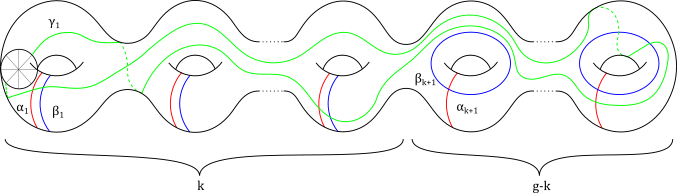}
	\caption{A non-orientable trisection diagram in $X_1$-standard position. We think of the surface as the connected sum of a single Klein bottle and $g-1$ $2$-tori. The wheel-like symbol denotes antipodal identification.}
	\label{fig:nonorientable_tris_diag}
\end{figure}

Using Algorithm \ref{alg:tris_to_kirby_prime} we can again introduce a ''guessing method'' to find trisection diagrams of closed non-orientable $4$-manifolds. In particular we get 

\begin{corollary}
There is no closed non-orientable $4$-manifold that has genus $0$. The only closed oriented $4$-manifolds that has genus $1$ is $\mathbb{S}^1\Tilde{\times}\,\mathbb{S}^3$.
\end{corollary}

\begin{proof}
Genus $0$ implies that the associated manifolds cannot have any $1$-handles and are therefore orientable.\par
To prove the classification for genus $1$ we use the fact that there are only $4$ isotopy classes of orientation preserving simple closed curves on the Klein bottle \cite{Lickorish1963}. Of these only one is both non-trivial and non-seperating: the meridian. We thus end up with a trisection diagram as in Figure \ref{fig:nonorientable_sphere_bundle} which describes $\mathbb{S}^1\Tilde{\times}\,\mathbb{S}^3$.
\end{proof}

\begin{figure}[!htb]
	\centering
	\includegraphics[width=0.5\linewidth]{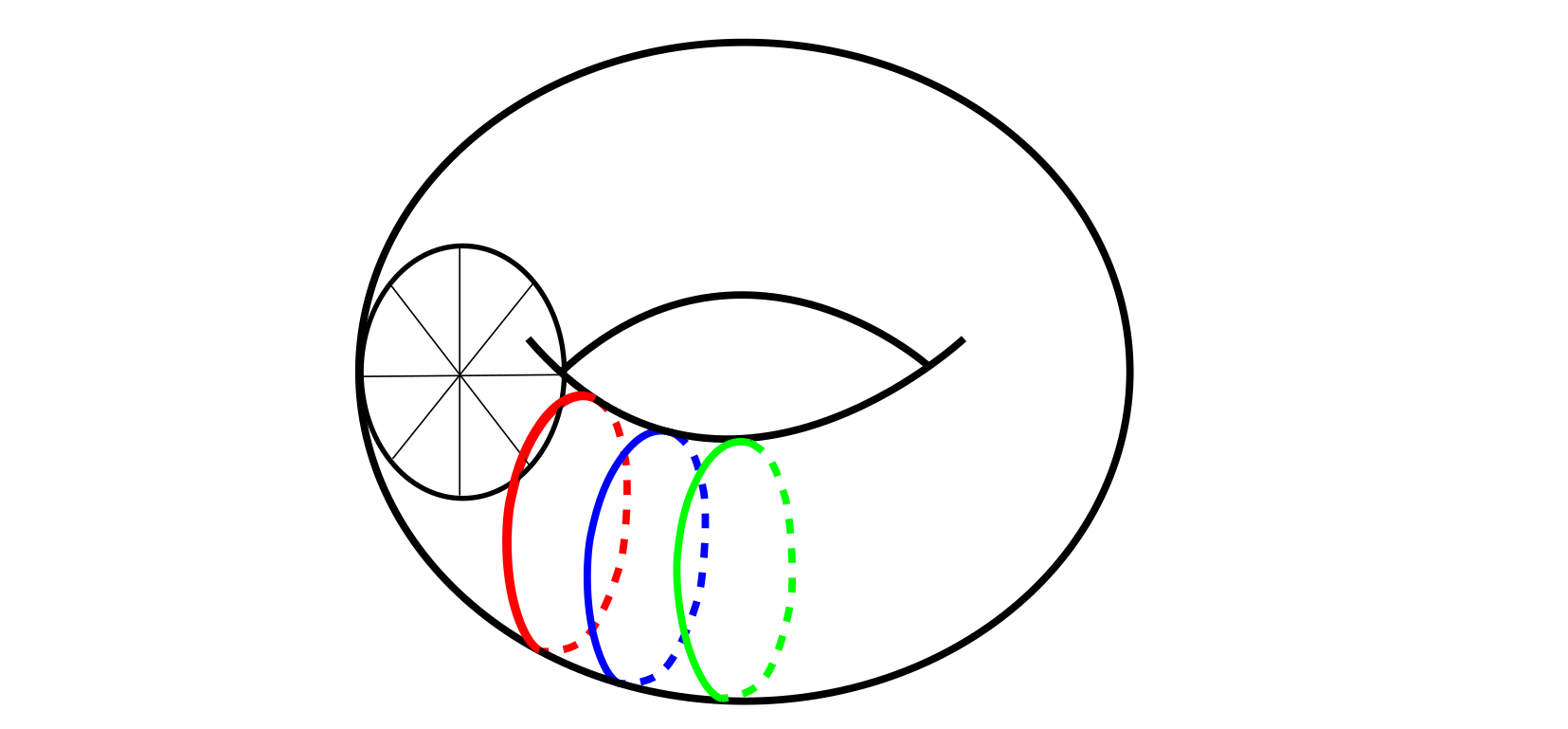}
	\caption{The genus $1$ trisection diagram of $\mathbb{S}^1\Tilde{\times}\,\mathbb{S}^3$}
	\label{fig:nonorientable_sphere_bundle}
\end{figure}

\begin{remark}
$\Tilde{\Sigma}_g$ for $g>1$ has, in stark contrast to the Klein bottle, infinitely many isotopy classes of nontrivial orientation preserving simple closed curves, making classification efforts for closed non-orientable $4$-manifolds of higher genus much more difficult. To see this we observe that there are infinitely many closed non-orientable $3$-manifolds of genus $2$ (for example the infinite family $(\mathbb{S}^1\Tilde{\times}\,\mathbb{S}^2)\sharp L(p;1)$, all of which are non-diffeomorphic to one another as they have different $\pi_1$) so cannot be only finitely many isotopy classes of nontrivial simple closed curves with orientable tubular neighbourhoods on $\Tilde{\Sigma}_g$ for $g>1$.
\end{remark}

Proposition \ref{prop:3d_diagram} can be straightforwardly adapted (except that we do not use the dotted circle notation), and, realizing that we can explicitly construct Heegaard splittings for $\sharp_g \mathbb{S}^1\Tilde{\times}\,\,\mathbb{S}^2$ like in the orientable case, we end up with 

\begin{manualalgorithm}{\ref{alg:main}'}\label{alg:main_prime}

Given a Kirby diagram describing a closed non-orientable $4$-manifold $X$, which comprises $k$ $1$-handles, one of which is orientation reversing, and a framed link $L$ consisting of $l$ knots $K_1,\dots,K_l$ one obtains a trisection diagram describing $X$ by applying the following algorithm (many of whose steps are illustrated in Figure \ref{intro_steps}):
\begin{enumerate}
    \item Replace the projection of the attaching regions of the $4$-dimensional $1$-handles by attaching regions of $3$-dimensional $1$-handles (a pair of $2$ disks) with a parallel $\alpha-\beta$ curve pair drawn around one of the disks. The attaching region of the orientation reversing $4$-dimensional $1$-handle should project to the attaching region of an orientation reversing $3$-dimensional $1$-handle.
    \item For every knot $K$ in the link $L$ that does not have an overcrossing, introduce a $+1$- then a $-1$-kink. In other words we do a Reidemeister-$2$ move between two strands of the same knot $K$.
    \item Introduce $\pm 1$ kinks matching the sign and numerical value of the blackboard framing.
    \item Replace intersections (including the ones introduced in step $1$) in the fashion shown in Figure \ref{intro_steps}.
    \item Do handle slides among the $\alpha$-curves such that $\lvert \alpha_i \cap K_j \rvert = \delta_{ij}$ for $i=1,\dots,g$ and $j=1,\dots,l$, where the $K_j$ are the knots constituting the link $L$ in the Kirby diagram, and $g$ is the genus of the surface $\Sigma$.
    \item Color $L$ green.
    \item Draw a green curve parallel to every unpaired $\alpha$-curve, i.e. all the $\alpha$-curves that do not intersect any of the $K_j$ (so precisely $\alpha_{l+1},\dots,\alpha_g$), and call these new green curves $\gamma_{l+1},\dots,\gamma_g$.
\end{enumerate}
\end{manualalgorithm}
Algorithm \ref{alg:main_prime} differs from Algorithm \ref{alg:main} only in step $1$.

\pagebreak
\bibliographystyle{alpha}
\bibliography{references}
\end{document}